\documentclass{lms}

\usepackage{latexsym}
\usepackage{amsfonts}

\usepackage{amsmath,amssymb}

\title[Cannon-Thurston fibers for iwip automorphisms of $F_N$]%
{Cannon-Thurston fibers for iwip automorphisms of $F_N$}

\author{Ilya Kapovich and Martin Lustig}

\classno{20F65 (primary),  57M, 37B, 37D (secondary)}

\extraline{The first author was partially supported by the NSF
  grant DMS-1405146 and by The Simons Foundation grant no. 279836. Both authors acknowledge support from U.S. National Science Foundation grants DMS 1107452, 1107263, 1107367 "RNMS: GEometric structures And Representation varieties" (the GEAR Network).}

\newtheorem{theor}{Theorem}

\newtheorem{thm}{Theorem}[section] 
\newtheorem{lem}[thm]{Lemma}
\newtheorem{cor}[thm]{Corollary} 
\newtheorem{prop}[thm]{Proposition}

\newnumbered{assertion}{Assertion}    
\newnumbered{conject}{Conjecture}  
\newnumbered{definition}{Definition}
\newnumbered{hypothesis}{Hypothesis}
\newnumbered{remark}{Remark}
\newnumbered{note}{Note}
\newnumbered{observation}{Observation}
\newnumbered{problem}{Problem}
\newnumbered{question}{Question}
\newnumbered{algorithm}{Algorithm}
\newnumbered{example}{Example}
\newnumbered{defn}[thm]{Definition}
\newnumbered{notation}[thm]{Notation}
\newnumbered{conv}[thm]{Convention}
\newnumbered{rem}[thm]{Remark}
\newnumbered{exmp}[thm]{Example}

\newnumbered{warning}[thm]{Warning}
\newnumbered{history}[thm]{Historical Note}
\newnumbered{openp}[thm]{Open Problem}

\newnumbered{prob}[thm]{Problem}
\newnumbered{defnrem}[thm]{Definition-Remark}
\newnumbered{aside}[thm]{Aside}
\newnumbered{noname}[thm]{}
\newnumbered{conventionfacts}[thm]{Conventions-Facts}

\def\strutdepth{\dp\strutbox}
\def \ss{\strut\vadjust{\kern-\strutdepth \sss}}
\def \sss{\vtop to \strutdepth{
\baselineskip\strutdepth\vss\llap{$\diamondsuit\;\;$}\null}}

\def\strutdepth{\dp\strutbox}
\def \sst{\strut\vadjust{\kern-\strutdepth \ssss}}
\def \ssss{\vtop to \strutdepth{
\baselineskip\strutdepth\vss\llap{$\spadesuit\;\;$}\null}}

\def\strutdepth{\dp\strutbox}
\def \ssh{\strut\vadjust{\kern-\strutdepth \sssh}}
\def \sssh{\vtop to \strutdepth{
\baselineskip\strutdepth\vss\llap{$\heartsuit\;\;$}\null}}

\newcommand{\R}{\mathbb R}

\newcommand{\Hy}{\mathbb H}

\def\epsilon{\varepsilon}
\def\phi{\varphi}

\def\hat{\widehat}
\def\bar{\overline}
\def\subset{\subseteq}

\newcommand{\Out}{\mbox{Out}}
\newcommand{\Aut}{\mbox{Aut}}

\newcommand{\card}{\mbox{card}}

\newcommand{\diag}{\mbox{diag}}

\newcommand{\normal}{\lhd}

\newcommand{\FN}{F_N}   

\newcommand{\cvnbar}{\overline{\mbox{cv}}_N}
\newcommand{\CVN}{\mbox{CV}_N}
\newcommand{\CVNbar}{\overline{\mbox{CV}}_N}

\newcommand{\ind}{\mbox{ind}_\mathcal Q}

\newcommand{\qed}{$\square$}

\begin{document}

\maketitle

\begin{abstract}
  For any atoroidal iwip $\phi \in \Out(\FN)$ the mapping torus group
  $G_\phi=F_N\rtimes_\phi \langle t\rangle$ is hyperbolic, and, by a
  result of Mitra, the
  embedding $\iota: \FN \overset{\lhd}{\longrightarrow} G_\phi$
  induces a continuous, $\FN$-equivariant and surjective {\em
    Cannon-Thurston map} $\hat \iota: \partial \FN \to \partial
  G_\phi$.

  We prove that for any $\phi$ as above, the map $\hat \iota$ is
  finite-to-one and that the preimage of every point of $\partial
  G_\phi$ has cardinality $\le 2N$.

  We also prove that every point $S\in \partial G_\phi$ with $\ge 3$
  preimages in $\partial F_N$ has the form $(wt^m)^\infty$ where $w\in
  F_N, m\ne 0$, and that there are at most $4N-5$ distinct
  $F_N$-orbits of such {\em singular} points in $\partial G_\phi$ (for
  the translation action of $F_N$ on $\partial G_\phi$).

  By contrast, we show that for $k=1,2$ there are uncountably many
  points $S\in \partial G_\phi$ (and thus uncountably many
  $\FN$-orbits of such $S$) with exactly $k$ preimages in $\partial
  F_N$.

\end{abstract}

\maketitle

\section{Introduction}

The notion of a Cannon-Thurston map goes back to a celebrated preprint
of Cannon and Thurston from 1984 that was eventually published in
2007~\cite{CT}. They consider a closed hyperbolic $3$-manifold $M$
which fibers over a circle, with the fiber being a closed hyperbolic
surface $\Sigma$. Then the inclusion $\Sigma\subseteq M$ lifts to the map
between their universal covers
$i:\widetilde \Sigma \to \widetilde M$, where $\widetilde \Sigma =\mathbb H^2$ and
$\widetilde M = \mathbb H^3$. Cannon and Thurston prove in \cite{CT}
that the map $i$ extends to a continuous $\pi_1(S)$-equivariant map between the hyperbolic
boundaries at infinity: $\hat \iota: \partial_\infty \mathbb H^2 \to \partial_\infty \mathbb H^3$, where $\mathbb
\partial_\infty \mathbb H^2 = \mathbb S^1$ and $\partial_\infty \mathbb H^3 = \mathbb S^2$. The
map $\hat \iota$ is necessarily surjective, and so, being a continuous map
from 
$\mathbb S^1$ to $\mathbb S^2$,
it gives a space-filling
curve. Moreover, the map $\hat \iota$ is finite-to-one, and the full preimage
of every point of $\mathbb S^2$ has cardinality at most $4g-2$, where
$g$ is the genus of the fiber $\Sigma$.

\smallskip

In group-theoretic terms, in this example we have an inclusion $H\le
G$, where $H=\pi_1(\Sigma)$ and $G=\pi_1(M)$ are both word-hyperbolic, and their Gromov boundaries agree with the corresponding hyperbolic boundaries at infinity:  $\partial H = \partial_\infty \Hy^2 = \mathbb S^1$ and $\partial G = \partial_\infty \Hy^3 = \mathbb S^2$.  The natural question about possible generalizations of the Cannon-Thurston result 
led to the
following definition (see 
subsection \ref{CTmaps} below for a more precise statement):

\smallskip

If $G$ is a word-hyperbolic group and $H$ a
word-hyperbolic subgroup, and if the inclusion $\iota: H\to G$ extends to a continuous map
$\hat \iota: \partial H \to \partial G$, then the map $\hat \iota$ is called the
\emph{Cannon-Thurston map}; in this context this definition is due to Mitra~\cite{M1,M2,M3}. 
In particular, if the Cannon-Thurston map $\hat \iota: \partial H \to
\partial G$ exists, then this map is unique and for any sequence
$h_n\in H\cup \partial H$ converging to some $X\in \partial H$ in the
topology of $H\cup\partial H$, we have $\lim_{n\to\infty} h_n=\hat
\iota(X)$ in $G\cup \partial G$. It is known, see
\cite[Proposition~2.12]{JKLO},  that if $H$ is a non-elementary word-hyperbolic subgroup of a word-hyperbolic group $G$, then a map $\partial H\to \partial G$ is the Cannon-Thurston map if and only if this map is continuous and $H$-equivariant.

\smallskip

It is well-known that, if $H\le G$ is a quasiconvex
subgroup of a word-hyperbolic
group $G$, then $H$ is word-hyperbolic and the inclusion $H\le G$
extends to a continuous topological embedding $\partial H\to\partial
G$. Thus in this case the Cannon-Thurston map exists and, moreover, is
injective.
Surprisingly, it turns out that the Cannon-Thurston map exists in many
situations where $H\le G$ is not quasiconvex, as shown by the work of
Mitra in 1990s~\cite{M1,M2,M3,M4}.

\smallskip

In particular, a result of Mitra~\cite{M2} states
that whenever
\[
1\to H \to G \to Q\to 1
\]
is a short exact sequence of word-hyperbolic groups, then the inclusion
$H\le G$ extends to a continuous Cannon-Thurston map $\hat \iota: \partial
H \to \partial G$. It is well-known~\cite{Alonso91} that in this
situation, if $H$ and $Q$ are infinite, then $H\le G$ is not
quasiconvex. Also, if $H$ 
is infinite, then the limit
set of $H$ in $\partial G$ is equal to $\partial G$~\cite{KS96} and
therefore the map $\hat \iota: \partial
H \to \partial G$ is onto.
This result of Mitra generalizes the original
theorem of Cannon and Thurston mentioned above, since in that context
one has a short exact sequence $1\to \pi_1(\Sigma)\to \pi_1(M)\to \mathbb Z\to 1$.

\smallskip

Until recently it has been unknown if there are any inclusions $H\le G$ (with $H$ and
$G$ word-hyperbolic) where the Cannon-Thurston map does not exist~\cite{M3,KB}. A
surprising new result of Baker and Riley~\cite{BR} constructs the first
example of such an inclusion (with $H=F_3$ ) where the Cannon-Thurston
map does not exist. Their results were subsequently further extended by Matsuda and Oguni~\cite{MO}.

\medskip

The result of Mitra, mentioned above, applies in particular to word-hyperbolic
free-by-cyclic groups. Recall that if $\Phi\in \Aut(F_N)$ is an
automorphism of $F_N$, then the \emph{mapping torus} group of $\Phi$
is
\[
G_\Phi=F_N\rtimes_\Phi \langle t\rangle =\langle F_N, t\mid t ht^{-1}=\Phi(h), h\in F_N\rangle. 
\]

An automorphism $\Phi$ of $F_N$ is called \emph{hyperbolic} if the
group $G_\Phi$ is word-hyperbolic. It follows from the Bestvina-Feighn
Combination Theorem~\cite{BF92} and a result of Brinkmann~\cite{Br}
that $\Phi\in\Out(F_N)$ is hyperbolic if and only if 
$\Phi$ is \emph{atoroidal}, that is, does not have any nontrivial
periodic conjugacy classes in $F_N$ (which is also equivalent to the
condition that $G_\Phi$ does not contain any $\mathbb Z\times\mathbb
Z$-subgroups). An element $\phi\in \Out(F_N)$ is called
\emph{hyperbolic} if some (equivalently, any) representative
$\Phi\in\Aut(F_N)$ of $\phi$ is hyperbolic.
It is easy to see that
$G_\Phi$ and the inclusion $F_N\le
G_\Phi$ depend only on the outer automorphism class $\phi$ of $\Phi$, so that 
for simplicity
we will write from now on
$G_\phi$ instead of $G_\Phi$.
So, if $\phi\in \Out(F_N)$ is a hyperbolic automorphism then we have
a short exact sequence
\[
1\to F_N\to G_\phi\to\langle t\rangle \to 1
\]
of three word-hyperbolic groups, and hence, 
as discussed above,
there
does exist a continuous $F_N$-equivariant surjective Cannon-Thurston map $\hat \iota:
\partial F_N\to\partial G_\phi$.  

\smallskip

By now the properties of the Cannon-Thurston map in the original context of~\cite{CT} of  a closed hyperbolic 3-manifold fibering over a circle are very well understood. By contrast, apart from its existence, little has been known about the specific properties of the Cannon-Thurston map for mapping torus groups of hyperbolic automorphisms of free groups. The most typical type of hyperbolic automorphisms of free groups are
so-called iwip or ``fully irreducible'' hyperbolic
automorphisms. Recall that an element $\phi\in\Out(F_N)$ is said to be
\emph{irreducible with irreducible powers} (\emph{iwip}, for short),
or \emph{fully irreducible}, if no positive power of $\phi$ preserves
the conjugacy class of a proper free factor of $F_N$. Bestvina and
Handel proved~\cite{BH92} that if an iwip $\phi\in \Out(F_N)$ fails to
be atoroidal (i.e., in view of the above discussion, fails to be
hyperbolic) then $\phi$ is induced by a homeomorphism of a compact
connected surface with a single boundary component. Thus, for $N\ge
3$, ``most'' iwips are atoroidal. By contrast, it is easy to see that for $N= 2$
there are no atoroidal elements in $\Out(F_2)$. Moreover, in a sense
made precise by Rivin~\cite{Rivin}, for $N\ge 3$ a ``random'' element of
$\Out(F_N)$ is a hyperbolic iwip.
Note also that, if $\phi\in \Out(F_N)$ is a hyperbolic iwip, then $\partial G_\phi$ is known (by combined results of \cite{Br02}, \cite{KK00}) to be homeomorphic to the Menger curve. 
As recently proved by Dowdall, Kapovich and Leininger in \cite{DKL}, for a hyperbolic $\phi\in \Out(F_N)$ being fully irreducible is equivalent to being \emph{irreducible}, in the sense originally defined by Bestvina and Handel in \cite{BH92}.

 If $\phi\in\Out(F_N)$ is a hyperbolic automorphism, for a point $S\in
\partial G_\phi$ let the \emph{degree of $S$}, denoted $\deg(S)$, be
the cardinality of the set $\hat \iota^{-1}(S)$. Since $\hat \iota$ is
surjective, for every $S\in \partial F_N$ we have $\deg(S)\ge 1$.

We can now state the first result of this paper, proved in Section \ref{sec:main-results} below:

\begin{theor}\label{thm:A}
Let 
$\phi\in \Out(F_N)$ be a hyperbolic iwip,
and let $G_\phi=F_N\rtimes_\phi \mathbb Z$ be the mapping torus group of $\phi$. 
Then for every $S\in \partial G_\phi$ we have:
$$\deg(S)\le 2N.$$
\end{theor}

Moreover, as noted below in Remark~\ref{sharp-bound}, the $2N$ bound in Theorem~\ref{thm:A} is sharp, that is, for every $N \geq 3$ there exist an automorphism $\phi$ as in Theorem~\ref{thm:A} such that for some $S\in \partial G_\phi$ we have $\deg(S)=2N$.  By showing that in Theorem~\ref{thm:A} the Cannon-Thurston map $\hat\iota:\partial F_N\to\partial G_\phi$,  Theorem~\ref{thm:A} provides a positive answer, for the case of mapping tori of hyperbolic iwips,  to Problem~1.20, attributed to Swarup, in Bestvina's Geometric Group Theory problem list~\cite{Be04}.

\smallskip

In~\cite{M1} Mitra gave a description of the fibers of the
Cannon-Thurston map $\hat \iota:\partial H\to\partial G$ for any short exact
sequence 
of three hyperbolic groups $1\to H\to G\to
Q\to 1$. This description is given in terms of ``ending laminations''
$\Lambda_z$, $z\in\partial Q$, where $\Lambda_z\subseteq \partial^2
H=\{(X,Y)\in \partial H\times\partial H: X\ne Y\}$.  Given a hyperbolic iwip
$\phi\in\Out(F_N)$, there are several 
``laminations'' $\subseteq \partial^2 F_N$ naturally associated to $\phi$ that arose in the
study of $\Out(F_N)$:
The laminations $L_{BFH}(\phi^{\pm 1})\subseteq \partial^2 F_N$ were introduced
by Bestvina, Feighn and Handel in \cite{BFH97} and are defined in terms of
train tracks representing $\phi$.  The laminations
$L(T_\pm(\phi))$ of the the trees $T_\pm(\phi)$
are special cases of the general notion of a "dual" or "zero" lamination $L(T)$ for 
an $\R$-tree
$T$ 
with isometric $\FN$-action
introduced in \cite{CHL2}. Here $T_\pm(\phi)$
define 
the attracting/repelling fixed points for
the (right) action of $\phi$ on the compactified Outer space
$\CVNbar$. In our earlier work~\cite{KL6} we showed that for a hyperbolic iwip
$\phi\in\Out(F_N)$ we have $L(T_-(\phi))=\diag(L_{BFH}(\phi))$, the
``diagonal extension''  of $L_{BFH}(\phi)$. See
Section~\ref{sect:laminations} below for precise definition of
these terms.

\smallskip

The first step in the proof of
Theorem~\ref{thm:A} is 
to relate, using our results from~\cite{KL6},  
Mitra's ``ending laminations'' $\Lambda_{\phi^{\pm 1}}$ , for the short exact sequence
corresponding to the mapping torus group of a hyperbolic iwip
$\phi\in\Out(F_N)$, to the laminations $L(T_\pm(\phi))$.  
We prove:

\medskip
\noindent
{\bf Proposition \ref{prop:step1}:}
{\em
Let $\phi \in \Out(\FN)$ be a hyperbolic iwip. Then
$$\Lambda_\phi=L(T_-(\phi))=\diag(L_{BFH}(\phi))\,.$$}

Then, by
Mitra's results from~\cite{M1}, Proposition~\ref{prop:step1} implies Corollary~\ref{cor:step1} which states
that for the Cannon-Thurston map $\hat \iota:\partial F_N\to\partial
G_\phi$ and for distinct $X,Y\in\partial F_N$ we have $\hat \iota(X)=\hat
i(Y)$ if and only if $(X,Y)\in L(T_-(\phi))\cup L(T_+(\phi))$.  Corollary~\ref{cor:step1}  is a key fact for our analysis of the fibers of the Cannon-Thurston map.
After obtaining Corollary~\ref{cor:step1}, we use a description, due to Coulbois, Hilion and Lustig in
\cite{CHL2}, of the dual lamination $L(T)$, where $T$ is an
$\R$-tree with dense $F_N$-orbits (e.g. $T=T_\pm(\phi)$) in terms of the
so-called $\mathcal Q$-map. We combine this description of
$L(T_\pm(\phi))$ with the results of the ``index'' theory for trees that define points
in $\CVNbar$ and elements of $\Out(F_N)$, particularly a theorem of Coulbois-Hilion \cite{CH10} which gives a bound
for the \emph{$\mathcal Q$-index} of $T_\pm(\phi)$, to derive the
conclusion of Theorem~\ref{thm:A}.

Proposition~\ref{prop:step1} corrects an error in Mitra's paper  \cite{Mitra99}  and can be used to fix a gap, created by that error, in the proof of one of the main results of \cite{Mitra99}, namely Theorem~3.4 there regarding quasiconvexity of certain kinds of finitely generated subgroups in mapping tori of hyperbolic iwips.  Mitra's Theorem~3.4 is relevant for the new result of Hagen and Wise \cite{HW} about cubulating  hyperbolic free-by-cyclic groups.  We explain how to correct the proof of Theorem~3.4 of \cite{Mitra99} in Appendix~\ref{A} at the end of this paper.

Proposition~\ref{prop:step1} and Corollary~\ref{cor:step1} are also related to the general results of Bowditch~\cite{Bow02} about hyperbolic  boundaries and the associated Cannon-Thurston maps for one-sided and two-sided "hyperbolic stacks" of hyperbolic metric spaces.

After proving Theorem~\ref{thm:A}, we undertake a more detailed study of the fibers of the Cannon-Thurston map.
In analogy to the classical Cannon-Thurston situation
we say that $S\in \partial
G_\phi$ is \emph{simple} if $\deg(S)=1$, that $S$ is \emph{regular} if
$\deg(S)=2$, and that $S$ is \emph{singular}
if $\deg(S)\ge 3$. It is straightforward to show that $\deg(S)=\deg(gS)$ for
any $S\in \partial G_\phi$ and $g\in G_\phi$.
The group $G=G_\phi$ acts on $\partial G_\phi$ by translations, and hence so does $F_N\le G_\phi$. When referring to $G$-orbits or $F_N$-orbits of points in $\partial G$, we will mean these translation actions. The $\FN$-orbit of $S\in \partial G_\phi$ will be denoted by $[S]_{\FN}$; as argued above, the degree $\deg([S]_{\FN})$ is well defined.
The following result (proved in section \ref{sec:main-results}) gives fairly precise information about the singular points in $\partial G_\phi$:

\begin{theor}\label{thm:B}
Let 
$\phi \in \Out(\FN)$
be a hyperbolic iwip and let $G_\phi$ be its mapping torus group.
Then:
\begin{enumerate}
\item Every singular point $S\in \partial G_\phi$ 
has the form $S=(wt^m)^\infty$ for some $w\in F_N$ and $m\ne 0$.
\item The number $\sigma$ of $F_N$-orbits of singular points in
  $\partial G_\phi$ is finite and satisfies $2\le \sigma\le 4N-5$. 
\item We have
\[
\sum
\left(\deg([S]_{\FN})-2\right)\le 4N-5
\]
where the sum is taken over all $F_N$-orbits $[S]_{F_N}$ of
singular points in $\partial G_\phi$.
\end{enumerate}
\end{theor}
Theorem~\ref{thm:B} implies that for every singular $S\in\partial G_\phi$ there exists a unique $g\in G_\phi$ such that $g$ is not a proper power and such that $g^\infty=S$; moreover, there are $\le 4N-5$ conjugacy classes of $g\in G$ with these properties.

We next summarize, in a simplified form, the remaining results (obtained in Section~\ref{sec:main-results}) about fibers of $\hat \iota$ for $G_\phi$.

\begin{theor}\label{thm:C}
Let 
$\phi \in \Out$ be a hyperbolic iwip and let $G_\phi$ be its mapping torus group.
Then the following hold:
\begin{enumerate}
\item Let $g=wt^m\in G_\phi$ where $w\in F_N$ and $m\ne 0$.
Then
\[
\deg(g^\infty)+ \deg( g^{-\infty})\le 4N-1.
\]
\item If $w\in F_N, w\ne 1$ then the point $w^\infty\in \partial G_\phi$ is simple.
\item There are uncountably many $G_\phi$-orbits of simple points in $\partial G_\phi$. (Since there are only countably many rational points in $\partial G_\phi$, this also implies that there are  uncountably many $G_\phi$-orbits of irrational simple points in $\partial G_\phi$.)
\item There are uncountably many $G_\phi$-orbits of regular points in $\partial G_\phi$. (Again, this also implies that there are uncountably many $G_\phi$-orbits of irrational regular points in $\partial G_\phi$).
\end{enumerate}

\end{theor} 

The results of this paper, together with the results of Dowdall, Kapovich and Leininger in \cite{DKL}, indicate that there is a possible interesting relationship between the Cannon-Thurston maps corresponding to different ways in which a given hyperbolic free-by-cyclic group $G_\phi$ splits as the mapping torus group of a free group automorhism.

Finally, we'd like to note that analogues and relatives of the Cannon-Thurston map have also been investigated in other contexts arising in the study of hyperbolic 3-manifolds and mapping class groups (e.g. see~\cite{Bow02,Bow07,Kla,LMS,McM,Min94,M09}), of relatively hyperbolic groups~\cite{Ger10,GP09,GP10,GP11,MP11}, and of the dynamics of complex polynomials (e.g. see~\cite{Kam03,Milnor04,Me09,Shi00}).

\medskip
\noindent

\begin{acknowledgements}\label{ackref}
The first author thanks Arnaud Hilion and Thierry Coulbois for useful
discussions regarding the $\mathcal Q$-index. The authors also thank
the refeee for useful suggestions.
\end{acknowledgements}

\section{Preliminraies}
\label{sect:outer}

\subsection{Iwip automorphisms of $\FN$}
\label{iwip-autos}

${}^{}$

Throughout this paper $\FN$ denotes the non-abelian free group of finite rank $N \geq 2$.  An automorphism $\Phi \in \Aut(\FN)$, or its associated outer automorphism $\phi \in \Out(\FN)$, is called 
{\em fully irreducible} or 
{\em iwip} (for {\em irreducible with irreducible powers}) if there is no non-trivial proper free factor of $\FN$ which is mapped by any positive power of $\Phi$ to a conjugate of itself. 

It follows directly that any such $\phi$ has infinite order, and any positive or negative power of $\phi$ is also iwip.

\smallskip

For any automorphism $\Phi: \FN \to \FN$ the semi-direct product
\[G_\Phi=\FN\rtimes_\Phi \langle t_\Phi\rangle=\langle \FN, t_\Phi \mid t_\Phi wt_\Phi^{-1}=\Phi(w)  {\rm \,\, for\,\, any\,\,} w\in \FN\rangle \tag{$\clubsuit$}
\]
is called the {\em mapping torus groups} defined by $\Phi$.  It is well known and easy to see that for any two $\Phi, \Phi' \in \Aut(\FN)$ which define the same outer automorphism $\phi \in \Out(\FN)$ one has 
$G_{\Phi'} \cong G_{\Phi}$. 
Indeed, since $\FN$ has trivial center, for $\phi \neq 1$ there is a canonical identification between $G_\Phi$ and the full preimage of the cyclic group $\langle \phi \rangle \subset \Out(\FN)$ under the quotient map $\pi: \Aut(\FN) \to \Out(\FN)$.
Hence 
we will denote 
the group $G_\Phi$ often by $G_\phi$.

The above identification $G_\phi = \pi^{-1}(\langle \phi \rangle)$ is also useful to understand the canonical extension of the $G_\phi$-action (by conjugation) on the normal subgroup $\FN \normal G_\phi$ to a $G_\phi$-action on the boundary $\partial \FN$. In particular, for any $X \in \partial \FN$ we obtain $t_\Phi(X) = \Phi(X)$.

\smallskip

\begin{rem}
\label{hyp-ator-surf}
For any iwip automorphism $\phi \in \Out(\FN)$ the equivalence of the following statements is well known 
(combined work of \cite{BF92} and \cite{Br}):

\begin{enumerate}
\item
$\phi$ is atoroidal (i.e. no positive power of $\phi$ fixes any non-trivial conjugacy class $[w] \subset \FN$).
\item
$\phi$ is not induced by a homeomorphism of a surface with boundary.
\item
The mapping torus group $G_\phi$ is word-hyperbolic.
\end{enumerate}
Note that, since any automorphism of $F_2$ is induced by a homeomorphism of the punctured torus, any iwip $\phi$, which satisfies the above three equivalent conditions, necessarily satisfies
$
N \geq 3 
$.
Furthermore, 
$\phi \in \Out(\FN)$ is a toroidal 
(= not atoroidal)
iwip if and only if $\phi^t$ is a toroidal iwip, for any integer $t \neq 0$.
\end{rem}

\subsection{The Cannon-Thurston map}
\label{CTmaps}

${}^{}$

For any atoroidal iwip $\phi \in \Out(\FN)$ the inclusion $\iota : \FN \overset{\normal}{\longrightarrow} G_\phi$ induces, by a more general result of Mitra \cite{M2}, a well defined {\em Cannon-Thurston map}
$$\hat \iota : \partial \FN \to \partial G_\phi$$
which is continuous and $\FN$-equivariant. Moreover, 
with respect to the above explained
$G_\phi$-action on $\partial \FN$, the Cannon-Thurston map $\hat \iota$ is easily seen to be actually $G_\phi$-equivariant.

From the fact that $\FN$ is 
infinite and 
normal in $G_\phi$, and hence $\hat\iota(\partial \FN)$ a non-empty and $G_\phi$-invariant subset of $\partial G_\phi$, one deduces:

\begin{prop}
\label{CT-surjectivity}
For any atoroidal iwip $\phi \in \Out(\FN)$ the Cannon-Thurston map
$\hat \iota : \partial \FN \to \partial G_\phi$ is surjective.
\qed
\end{prop}

\subsection{$\R$-trees and iwip automorphisms}
\label{trees-graphs}

${}^{}$

$\R$-trees $T$ with isometric $\FN$-action have become the object of much research in the past 30 years; one usually assumes that the tree $T$ is {\em minimal} (i.e. there is no $\FN$-invariant proper subtree in $T$). 
The group $\Out(\FN)$ acts properly discontinuously on {\em Outer space} $\CVN$, which consists of projective classes $[T]$ of such $\R$-trees $T$, with the additional specifications that the $\FN$-action on $T$ is free and 
discrete.
The action of $\Out(\FN)$ extends to the compactification $\CVNbar$, which still consists of projective classes of $\R$-tree actions, but without the last two specifications.

More specifically, the space $\CVNbar$ is the quotient of the ``unprojectivized'' space $\cvnbar$ of {\em very small} $\R$-trees $T$.
Every $T\in \cvnbar$ is uniquely
determined by its \emph{translation length function}:
$$||\cdot||_T:F_N\to\mathbb R, \,\,\,
||g||_T := \inf_{x \in T} d(x, gx) $$
Two trees $T_1,T_2\in \cvnbar$ are close if the functions
$||\cdot||_{T_1}$ and $||\cdot||_{T_1}$ are pointwise close on a large ball in
$F_N$. 
For more details see \cite{CV,Gui,Gui1,Vog}.
A tree $T \in \cvnbar$ is said to 
have {\em dense orbits} if the $\FN$-orbit of some (or equivalently, of any) $x \in T$ is dense in $T$.

\smallskip

For any $\R$-tree $T$ we denote by $\bar T$ its metric completion, and by $\partial T$ its Gromov boundary. The $\FN$-action on $T$ extends canonically 
the union $\hat T := \bar T \cup \partial T$. In \cite{CHL} a slight weakening of the metric topology on $\hat T$ has been introduced, the so-called {\em oberservers' topology}; on any segment $[x, y] \subset T$ the two topologies agree.

\begin{prop}~\cite{LL,CHL}
\label{Q-map}
Let $T \in \cvnbar$ be an $\R$-tree dense orbits.  

\smallskip
\noindent
(1) Then there exists a surjective $\FN$-equivariant map $\mathcal Q: \partial \FN \to \hat T$ which is continuous with respect to the observers' topology on $\hat T$ (but in general not with respect to the metric topology).

\smallskip
\noindent
(2) Furthermore, for any $P \in T$ the map $\mathcal Q$ 
arises from extending continuously
(with respect to the observers' topology) 
the map $\mathcal Q_{P}: \FN \to T, \,\,w \mapsto wP$,
and as such $\cal Q$ is unique. 
\qed
\end{prop}

Any iwip $\phi \in \Out(\FN)$ acts on $\CVNbar$ with locally uniform North-South dynamics (see \cite{LL}), and the two projectively fixed trees on the {\em Thurston boundary} $\partial \CVN := \CVNbar \smallsetminus \CVN$, called  $T_+ = T_+(\phi)$ and $T_- = T_-(\phi)$ both have the property that the $\FN$-action is free, and that they have dense orbits.

\smallskip

The fact that both $T_+$ and $T_-$ are projectively fixed by $\phi$ translates, 
for any lift $\Phi \in \Aut(\FN)$ of $\phi$, 
into the existence of  homotheties $H_+: T_+ \to T_+$ and $H_-: T_- \to T_-$ with stretching factors $\lambda_+ > 1$ and $\frac{1}{\lambda_-} < 1$ respectively, which {\em realize} 
$\Phi$ in the following sense:

For any $w \in \FN$ and any $x \in T_\delta$ (for $\delta = +$ or $\delta = -$) one has $H_ \delta(w x) = \Phi(w) H_ \delta(x)$, or equivalently
$$ \Phi(w) = H_ \delta w H_ \delta ^{-1}: T_ \delta \to T_ \delta\, .$$
In this case, the action of $\FN$ on $T_\delta$ by isometries extends canonically to an action of $G_\phi$ by homotheties, by defining $t_\Phi x = H_ \delta(x)$ for any $x \in T_ \delta $. 
As above for the $\FN$-action, the $G_\phi$-action too extends naturally to $\hat T_ \delta $. 
Part (2) of Proposition \ref{Q-map} implies directly the following:

\begin{prop}
\label{G-equivariance}
For any atoroidal iwip $\phi \in \Out(\FN)$ 
the two maps $\mathcal Q_+: \partial \FN \to \hat T_+$ and $\mathcal Q_-: \partial \FN \to \hat T_-$ are $G_\phi$-equivariant.
\qed
\end{prop}

\subsection{The $\mathcal Q$-index}

${}^{}$

Coulbois and Hilion introduced in  \cite{CH10} the notion of a
\emph{$\mathcal Q$-index} for $\R$-trees with isometric $\FN$-action and dense orbits.  
They first define a local $\mathcal Q$-index for any point $x \in \bar T$; their definition involves also the stabilizer in $\FN$ of $x$. Since we are here only concerned with free actions, we restrict ourselves to this case, which simplifies 
things 
considerably. In this case their definition amounts to:
$$\ind(x) := \card (\mathcal Q^{-1}(x)) - 2$$

Since the map $\mathcal Q$ is $\FN$-invariant, the $\mathcal Q$-index is an invariant of the $\FN$-orbit $[x]_{F_N}$ of $x$, so that the term $\ind([x]_{F_N}) := \ind(x)$ is well defined. The summation over the $\FN$-orbits with non-negative index gives the $\mathcal Q$-index of $T$; however, it should be pointed out that the summation has to be taken over all $\FN$-orbits in the metric completion $\bar T$ of $T$ and not just in $T$.

\begin{defn}
Let $T$
be an $\R$-tree with isometric $\FN$-action which 
is free and has dense orbits. 
The \emph{$\mathcal Q$-index} 
of $T$ is defined as follows:
\[
\ind(T):=\sum_{[x]_{F_N} \in \bar T/\FN}
\max\{0,\ind ([x]_{F_N})\}. 
\]
\end{defn}

The following important general fact was recently established by Coulbois and
Hilion in~\cite{CH10}.

\begin{prop}\label{prop:Q-ind}
Let $T\in \cvnbar$ be a tree $T$ with dense orbits.
Then one has:
$$\ind(T)\le 2N-2$$
\end{prop}

\section{Algebraic laminations}\label{sect:laminations}

\subsection{Basic facts and definitions}

${}^{}$

As before, let $\FN$ be the free group of rank $N \geq 2$. 
We denote by
\[
\partial^2 F_N:=\{(X,Y) \mid X,Y\in \partial F_N, \text{ and } X\ne Y\}
\]
the {\em double boundary} of $\FN$.
As a subspace of  $\partial F_N\times
\partial F_N$ one inherits on $\partial^2 F_N$
the induced topology.
The left translation action of $F_N$ on
$\partial F_N$ induces a natural diagonal action of $F_N$ on
$\partial^2 F_N$ by homeomorphisms.
The space $\partial^2 F_N$ comes equipped with the canonical ``flip''
map 
given by $(X,Y)\mapsto (Y,X)$ for any $(X,Y)\in \partial^2 F_N$.

\begin{defn}\label{defn:algebraic_lamination}
An \emph{algebraic lamination} 
is a closed 
$F_N$-invariant and flip-invariant subset $L\subseteq \partial^2
F_N$. 
We also require $L$ to be non-empty.

If $L\subseteq \partial^2 F_N$ is an algebraic lamination, and $L_0\subseteq L$, we say that $L_0$ is a \emph{sublamination} of $L$ if $L_0\subseteq \partial^2 F_N$ is itself an algebraic lamination.

For $X,Y\in \partial F_N$ such that $(X,Y)\in L$ we say that $(X,Y)$
is a \emph{leaf} of $L$. For $X\in \partial F_N$ we say that $X$ is a
\emph{half-leaf} of $L$ if there exists $Y\in \partial F_N$
such
that $(X,Y)\in L$.
We denote by $L^1$ the set of half-leaves of the lamination $L$.

\end{defn}

Algebraic laminations have been introduced and studied in \cite{CHL1}; some background material for the use of laminations in our context can also be found in \cite{KL6}.

Any element $w \in \FN \smallsetminus \{1\}$ defines an algebraic lamination
$$
L_w = \FN \cdot (w^\infty, w^{-\infty}) \,\cup\, \FN \cdot (w^{-\infty} , w^\infty)
$$
where we mean by $w^\infty \in \partial \FN$ the 
limit of the elements $w^k$ for $k \to \infty$.
Clearly, the {\em rational} lamination $L_w$ depends only on the conjugacy class $[w] \subset \FN$ of $w$.

\begin{rem}
\label{basis-approach}
(1)
Whenever one fixes a basis $\mathcal A$ of the free group $\FN$ one obtains a canonical identification between the group $\FN$ and the set $F(\mathcal A)$ of reduced words in $\mathcal A \cup \mathcal A^{-1}$, which extends to an identification between $\partial \FN$ and the set of infinite reduced words $\partial F(\mathcal A)$. When working with laminations, the combinatorial objects from $F(\mathcal A)$ or $\partial F(\mathcal A)$ have many advantages and are often simply more concrete to work with; however, a basis free approach has the advantage of greater conceptual clarity.  In the sequel we will freely pass from one viewpoint to the other, as the transition is indeed canonical. 

\smallskip
\noindent
(2)
For example, the above defined point 
$w^\infty \in \partial \FN$ corresponds to the reduced infinite eventually periodic word which is obtained from reducing the infinite periodic word $w w w \ldots$.

\smallskip
\noindent
(3)
Similarly, the combinatorial object corresponding to a pair $(X, Y) \in \partial^2\FN$ is the biinfinite reduced word 
$Y_\mathcal A^{-1} X_\mathcal A$, where $X_\mathcal A$ and $Y_\mathcal A$ are the reduced infinite words in $\mathcal A^{\pm 1}$ that represent $X$ and $Y$ respectively, $Y_\mathcal A^{-1}$ is the left-infinite reduced word obtained by ``inverting'' $Y_\mathcal A$, and $Y_\mathcal A^{-1} X_\mathcal A$ is obtained from the ``product'' $Y_\mathcal A^{-1} \cdot X_\mathcal A$ by reduction at the multiplication locus.
\end{rem}

\begin{defn}
\label{lam-generated}
(a)
For any infinite set $\Omega$ of conjugacy classes $[w_i] \in \Omega$ the lamination $L(\Omega)$ {\em generated} by $\Omega$ is given as the set of accumulation points of 
the union of all $L_{w_i}$ (where ``accumulation points'' is meant in the classical meaning for a subset of a topological space).

Alternatively, for any fixed basis $\mathcal A$ of $\FN$, the lamination $L(\Omega)$ consists precisely of all leaves $(X, Y)$ such that any finite subword of the reduced biinfinite word $Y_\mathcal A^{-1} X_\mathcal A$ is also a subword of one of the reduced cyclic words $\hat w_i$ which represent  $[w_i] \in \Omega$, or of $\hat w_i^{-1}$.

\smallskip
\noindent
(b)
Similarly, for any boundary point $X \in \partial \FN$ we define the lamination $L(X)$ {\em generated by $X$} 
as the intersection of all laminations $L(\Omega)$, where $\Omega$ is the set of conjugacy classes that is given by any family of elements $w_k \in \FN$ which satisfy $\lim w_k = X$.

Again, for any fixed basis $\mathcal A$ of $\FN$ one can define $L(X)$ alternatively as the set
of all leaves $(Y, Z)$ such that any finite subword of the reduced biinfinite word $Z_\mathcal A^{-1} Y_\mathcal A$ is also a subword of $X_\mathcal A$ or of $X_\mathcal A^{-1}$.

\end{defn}

\begin{rem}
\label{minimal-laminations}
It follows directly that for any algebraic lamination $L$ the following are equivalent:
\begin{enumerate}
\item
$L$ is minimal with respect to the inclusion.
\item
For any leaf $(X, Y) \in L$
the union of the two orbits $\FN \cdot (X, Y) \,\cup\, \FN \cdot (Y, X)$ is dense in $L$.  
\item
Any half-leaf $X$ of $L$ generates $L$, i.e. $L = L(X)$.
\end{enumerate}
As a consequence we see that two minimal laminations 
$L_1, L_2$ are either equal, or else they are disjoint, with disjoint sets of half leaves:  
$$L_1^1 \cap L_2^1 = \emptyset$$
\end{rem}

\begin{rem}
\label{ends-of-L}
For any lamination $L$ a boundary point $X \in \partial \FN$ is called an {\em end} of $L$ if $X$ satisfies:
$$L(X) \subset L$$
The set of ends of $L$ is denoted by ${\rm Ends}(L)$.
We would like to warn the reader that even for a minimal lamination $L$ there exist boundary points $X \in \partial \FN$ which are ends but not half-leaves of $L$.
\end{rem}

The following statements for minimal laminations are well known in the context of symbolic dynamics; we only indicate the arguments:

\begin{rem}
\label{laminations-folk}
(1)
Every minimal lamination which is not rational contains a {\em singular} leaf, i.e.  there are two distinct leaves $(X, Z), (Y, Z) \in L$ which have a common half-leaf $Z$.

This is shown by first observing that every not eventually periodic half-leaf $X$, written as infinite reduced word $X_\mathcal A = x_1 x_2 \ldots$ in some basis $\mathcal A$ of $\FN$, contains arbitrary large ``special subwords'' $x_{k, m}:= x_k \ldots x_m$, i.e. 
there exist indices $k', m'$ with $x_{k, m} = x_{k', m'}$, 
$k \neq k'$ and $x_{m+1} \neq x_{m'+1}$. One then uses the finiteness of the set of words of any given length to find a nested sequence of such special subwords, and the fact that $L \subset \partial^2\FN$ is closed to construct the singular leaf.

\smallskip
\noindent
(2) 
The set of half-leaves of any minimal non-rational lamination $L$ is uncountable.

Again, one uses the existence of special subwords on every half-leaf and a standard diagonal argument to get uncountability.

\smallskip
\noindent
(3) 
For any finite set $L_1, \ldots L_k$ of minimal laminations there exist uncountably many points in $\partial \FN$ which are not half-leaves of any of the $L_i$.

This follows from the fact that there are infinitely many distinct minimal non-rational laminations in $\partial^2 \FN$, each of them has uncountably many half-leaves (by fact (2) above), and no two of them have a common half-leaf (by Remark \ref{minimal-laminations} (3)).
\end{rem}

\subsection{The diagonal extension}
\label{subsection-diagonal-ext}

${}^{}$

\begin{defn}[Diagonal extension]
\label{diagonal-extension}
For any subset $R\subseteq \partial^2 F_N$ the \emph{diagonal extension} of $R$,
denoted $\diag(R)$, is defined as:
\begin{gather*}
\diag(R)=\{(X,Y)\in \partial^2F_N \mid \text{there exits } 
Z_0=X, Z_1, \dots, Z_m=Y \\
\text{ such that } (Z_{i},Z_{i+1})\in R \text{ for } 0\le i\le m-1\}
\end{gather*}
If $R$ 
satisfies $R = \diag(R)$, then we say that $R$ is {\em diagonally closed}.
\end{defn}

Note that the definition of $\diag(R)$ is purely set-theoretic. In particular, for $R\subseteq \partial^2 F_N$ the set $\diag(R)$ need not be closed in $\partial^2 F_N$, and a diagonally closed subset of $\partial^2 F_N$ need not be a closed subset.
Clearly $R' \subset R$ implies $\diag(R') \subset \diag(R)$.
Using $m=1$ in the above definition we see that $R\subseteq \diag(R)$
for every $R\subseteq \partial^2 F_N$. 

\begin{lem}
\label{intersection-total}
Let $R = R_1 \cup R_2$ be the union of two sets $R_1, R_2 \subset \partial^2\FN$. Assume that $R$ is diagonally closed, and that $\diag(R_1) \cap \diag(R_2) = \emptyset$. Then both, $R_1$ and $R_2$ must also be diagonally closed.
\end{lem}

\begin{proof}
Any leaf $(X, Y) \in \diag(R_1) \smallsetminus R_1$ must be contained in $\diag(R) = R = R_1 \cup R_2$, and hence in $R_2 \subset \diag(R_2)$. This contradicts the assumption $\diag(R_1) \cap \diag(R_2) = \emptyset$. Hence $\diag(R_1) \smallsetminus R_1$ must be empty, or equivalently:  $R_1$ is diagonally closed. By symmetry the same applies to $R_2$.
\end{proof}

Note that, a priori, if
$L\subseteq \partial^2 F_N$ is an algebraic lamination, then $\diag(L)$
need not be an algebraic lamination, since $\diag(L)$ may fail to be a
closed subset of $\partial^2 F_N$.

\begin{rem}
\label{equal-half-leaves}
If $L$ and $\diag(L)$ are both algebraic laminations, then their sets of half-leaves are equal:
$$
L^1 = \diag(L)^1
$$
This follows directly from the above definition of the diagonal extension.
\end{rem}

Let $R_1, R_2 \subset \partial^2\FN$ be two disjoint sets. Then it is quite possible that $\diag(R_1) \cap \diag(R_2)$ is non-empty.  However, it follows directly from Definition \ref{diagonal-extension} that in this case $R_1$ and $R_2$ must have a common half-leaf $X \in \partial \FN$, i.e. there must be further elements $Y, Z \in \partial \FN$ such that $(X,Y) \in R_1$ and $(X,Z) \in R_2$.

\begin{lem}
\label{disjoint-extensions}
Let $L$ and $L'$ two distinct minimal lamination over $\FN$. Then one has:
$$\diag(L) \cap \diag(L') = \emptyset.$$
\end{lem}

\begin{proof}
Since $L$ 
is minimal, it follows (see Remark \ref{minimal-laminations} (3)) that 
laminations are either equal or disjoint.  Furthermore, 
for any half-leaf $X$
of $L$ the lamination $L(X)$ generated by $X$ is 
equal to $L$. The same is true for $L'$.  

We observed above that $\diag(L)$ and $\diag(L')$ are either disjoint, or else $L$ and $L'$ have a common half-leaf $X \in \FN$. Thus we obtain that $\diag(L) \cap \diag(L') \neq \emptyset$ implies $L = L(X) = L'$.
\end{proof}

\begin{rem}
\label{preimage-lamination}
The following assertions are direct consequences of the above definitions.

\smallskip
\noindent
(a)
Let $B$ be any set and
let $j: \partial \FN \to B$ be any map. Then the set 
$$L(j) := \{ (X, Y) \in \partial^2\FN \mid j(X) = j(Y) \}$$
is diagonally closed, that is, $L(j)=\diag(L(j))$. Note that this is a set-theoretic fact, and the set $B$ need not be a topological space here, and the map $j$ need not be continuous.

\smallskip
\noindent
(b) If $B$ is a topological space endowed with an $F_N$-action by homeomorphisms, and if  $j: \partial \FN \to B$ is a continuous $F_N$-equivariant map which is not injective, then the set 
$L(j)\subseteq \partial^2 F_N$ is an algebraic lamination (which is diagonally closed). 
\end{rem}

\subsection{The dual lamination of an $\R$-tree}

${}^{}$

In~\cite{CHL2}
the ``dual'' or ``zero'' lamination $L(T)$ of an $\R$-tree $T$
has been defined and investigated:

\begin{defn}
\label{L2}
Consider any $\R$-tree $T \in\cvnbar$.
\begin{enumerate}
\item
For any $\epsilon > 0$ let $\Omega_\epsilon(T)$ be the set of conjugacy classes 
$[w] \subset \FN \smallsetminus \{ 1 \}$ with translation length $|| w ||_T \leq \epsilon$, and let $L_\epsilon(T)=: L(\Omega_\epsilon(T))$ (for the notation see Definition \ref{lam-generated} (a)).
\item
 Define $L(T) := {\underset{\epsilon > 0}{\bigcap}} L_\epsilon(T)$.
\end{enumerate}
\end{defn}

For the reader who prefers the ``hands on'' combinatorial approach through fixing a base $\mathcal A$ of $\FN$, the lamination $L(T)$ can be described alternatively as the set of leaves $(X, Y) \in \partial^2\FN$ which have the property that for any $\epsilon > 0$ and any finite subword $v$ of the reduced biinfinite word $Y_\mathcal A^{-1} X_\mathcal A$ (compare Remark \ref{basis-approach} (3)) there is an element $w \in \FN$ with translation length $||w||_T \leq \epsilon$ such that the corresponding cyclically reduced cyclic word $\hat w$ contains $v$ as subword.

\begin{lem}
\label{subset-dual-lam}
Let $T \in \cvnbar$, and let $\Omega$ be an infinite set of conjugacy classes $[w_i] \subset \FN$ with the property that $\underset{ i \to \infty}{\lim}||w_i ||_T = 0$. Then the lamination generated by $\Omega$ satisfies:
$$L(\Omega) \subset L(T)$$
\end{lem}

\begin{proof}
This follows directly from 
Definition \ref{L2}, since the hypothesis $\underset{ i \to \infty}{\lim}||w_i ||_T = 0$ implies
(see Definition \ref{lam-generated} (a)) 
that $L(\Omega) \subset L_\epsilon(T)$ for any $\epsilon > 0$.
\end{proof}

If $T$ has dense orbits (see subsection \ref{trees-graphs}), then there is an alternative description of $L(T)$ in terms of the map $\mathcal Q: \partial \FN \to \hat T = \bar T \cup \partial T$
from Proposition \ref{Q-map}:

\begin{prop}
\label{prop:Q}
Let 
$T\in \cvnbar$ be a tree with dense $F_N$-orbits.
Then one has: 

\smallskip
\noindent
(1)\cite[Proposition 8.5]{CHL2}
For $X,Y\in \partial F_N$, $X\ne Y$ we have $(X,Y)\in L(T)$ if
and only if $\mathcal Q(X)=\mathcal Q(Y)$. 

\smallskip
\noindent
(2)\cite[Proposition 5.8]{CHL2}
If $X\in \partial
F_N$ and 
$\mathcal Q(X)\in \bar T$ then one has $L(X) \subset L(T)$
(i.e. $X$ is an end of $L(T)$, see Remark \ref{ends-of-L}).

\smallskip
\noindent
(3)\cite[Lemma 3.5]{GJLL}, \cite[Proposition 5.2]{CH10}
If $X\in \partial
F_N$ and 
$P:=\mathcal Q(X)\in \partial T$ then $\mathcal
Q^{-1}(P)=\{X\}$ holds.
\qed
\end{prop}
Note that Lemma~3.5 in \cite{GJLL} uses older terminology than the currently standard one. For the explanation of the transition to the terminology presently in use see Proposition~3.1 of \cite{LL}  together with the paragraph before and after it. 

\begin{rem}
\label{ends-halfleaves}
From parts (2) and (3) of the last proposition we obtain the following inclusions:
$$L(T)^1 \subset \mathcal Q^{-1}(\bar T)  \subset {\rm Ends}(L(T))$$
However, the reader should be warned that in general (including in the case $T = T_\pm(\phi)$) both of these inclusions are strict.
\end{rem}

We thus obtain as a direct consequence of Remark \ref{preimage-lamination} (a):

\begin{prop}\label{prop:tc}\label{CHL2}
Let 
$T\in \cvnbar$ be an $\R$-tree with 
dense
orbits. Then $L(T)$ is diagonally closed:
$$L(T)=\diag(L(T)).$$
\qed
\end{prop}

\begin{defn}
\label{laminations-equivalence}
For any algebraic lamination $L$ we consider the {\em associated equivalence relation} $\sim_L$, by which we mean the equivalence relation on $\partial \FN$ which is generated by the relation:
$$X \sim Y \,\, : \Longleftrightarrow \,\, (X, Y) \in L. $$
\end{defn}

\begin{prop}
\label{compact-quotient}
Let $L \subset \partial^2\FN$ be any diagonally closed algebraic lamination, and let $\sim_L$ be the associated equivalence relation on $\partial \FN$. 

Then the quotient set $\partial \FN /\sim_L$ provided with the quotient topology is a compact Hausdorff space.
\end{prop}

\begin{proof}
By definition of the diagonal closure
the set $L = \diag(L)$ is equal to the transitive closure of $L$ in $\partial^2 \FN$.
Since furthermore $L$ is flip-invariant, the subset $L \cup \{(X,X)\mid X \in \FN\} \subset \partial \FN \times \partial \FN$ defines a relation that is reflexive, symmetric and transitive, so that it must agree with the graph of the equivalence relation $\sim_L$ generated by $L$.

But as lamination $L$ is a closed subset of $\partial^2\FN$, which means precisely that $L \cup \{(X,X)\mid X \in \FN\}$ is closed in $\partial \FN \times \partial \FN$.

Therefore $\partial F_N/\sim_L$, with the quotient topology, inherits from $\partial \FN$ that it is a compact Hausdorff space.
\end{proof}

\begin{rem}
\label{mod-L(T)}
For any $\R$-tree $T \in \cvnbar$ with dense orbits we can consider the zero lamination $L(T)$ and the associated equivalence relation $\sim_{L(T)}$. It has been shown in \cite{CHL} that in this case the quotient space $\FN/\sim_{L(T)}$ is precisely the completed tree $\hat T$, equipped with the observers' topology, and the quotient map $\FN \to \FN/\sim_{L(T)}$ is precisely the map $\mathcal Q$, see subsection \ref{trees-graphs}.

\end{rem}

\subsection{Bestvina-Feighn-Handel laminations}\label{sect:iwip laminations}

${}^{}$

In~\cite{BFH97} Bestvina, Feighn and Handel introduced for every iwip automorphism $\phi \in \Out(\FN)$ a {\em stable} lamination which we denote by $L_{BFH}(\phi)$. This algebraic lamination was defined by the use of train track maps that represent $\phi$:
very roughly, it arises from iterating the train track map on any edge and passing to the limit.
For more details, also concerning the following proposition, the reader is referred to \cite{KL6}, in particular to its subsection 3.6. It is shown 
there that
the leaves of $L_{BFH}(\phi)$ have a uniform expanding property under iteration of $\phi$, while those of $L_{BFH}(\phi^{-1})$ are uniformly contracting.
 
\begin{prop}
\label{minimal-distinct}
Let $\phi \in \Out(\FN)$ be iwip. Then we have:
\begin{enumerate}
\item \cite[Proposition 1.8]{BFH97}, \cite[Proposition 3.38]{KL6}
The lamination $L_{BFH}(\phi)$ is minimal. 
\item \cite[Lemma 3.5]{BFH97}
The laminations $L_{BFH}(\phi)$ and $L_{BFH}(\phi^{-1})$ are distinct (and thus disjoint, by Remark \ref{minimal-laminations}).
\end{enumerate}
\qed
\end{prop}

\smallskip

In~\cite{KL6} we established the precise relationship between
$L_{BFH}(\phi)$ and $L(T_-(\phi))$. 
This result
has been subsequently generalized by Coulbois, Hilion and Reynolds~\cite{CHP11} to arbitrary ``indecomposable'' (in the sense of \cite{Gui3}) trees $T\in \cvnbar$.

\begin{thm}
\label{prop:KL6}
Let $\phi\in \Out(F_N)$ be an atoroidal iwip.
Then we have:
\[
L(T_-(\phi))=\diag\left(L_{BFH}(\phi)\right)
\]
In particular, $L_{BFH}(\phi)$ is the only minimal sublamination of $L(T_-(\phi))$.
\qed
\end{thm}

\smallskip

Via Proposition \ref{minimal-distinct} and Lemma \ref{disjoint-extensions},  the last proposition directly implies the following fact, 
previously considered ``folk knowledge'', namely
that for any atoroidal iwip $\phi$ the laminations
$L(T_+(\phi))$ and $L(T_-(\phi))$ are disjoint in the following strong sense:

\begin{prop}\label{prop:disj}
Let 
$\phi\in \Out(F_N)$ be an atoroidal iwip. 
Then
$L(T_+(\phi))\cap L(T_-(\phi)) =\emptyset$. Moreover, if $(X,Y)\in L(T_+(\phi))$ then
there does not exist $Z\in \partial F_N$ such that $(X,Z)\in L(T_-(\phi))$; that is, the laminations $L(T_+(\phi))$ and $L(T_-(\phi))$ have no common half-leaves.
\qed
\end{prop}

\section{Mitra's lamination}

In \cite{M1} Mitra gives, in a more general context, and with a slightly different vocabulary than used here, a definition which translates to the following:

\begin{defn}\label{defn:Mitra-lamination}
\label{Mitra-defn}
For any $h \in \FN \smallsetminus \{1\}$ let
$$\Lambda^{\pm}_{\phi, h} := L(\{\phi^n([h]) \mid n \geq 0 \})$$
and
$$\Lambda_\phi:= \underset{h\in F_N \smallsetminus \{ 1\}}{\cup} \Lambda^\pm_{\phi,h}\, .$$
\end{defn}

If we fix a basis $\mathcal A$ of $\FN$, then the laminations $\Lambda^{\pm}_{\phi, h}$ consist precisely of those leaves $(X, Y) \in \partial^2\FN$ which have the following property:  For any finite subword $v$ of the biinfinite reduced word $Y_\mathcal A^{-1} X_\mathcal A$ (see Remark \ref{basis-approach} (3)) exists an iterate $\phi^n$ with $n \geq 0$ such that $v$ or $v^{-1}$ is subword of the cyclically reduced cyclic word $\hat h_n$ which represents the conjugacy class $\phi^n([h])$.

\begin{rem}
\label{Mitra-comment}
(In \cite{M1} Mitra doesn't quite use the lamination $\Lambda^{\pm}_{\phi, h}$ as defined above, but rather works with a set $\Lambda_{\phi, h}$ which is close to $\Lambda^{\pm}_{\phi, h}$ but isn't quite an algebraic lamination in our sense (as he omits in the definition the ``or in $v^{-1}$'' from the previous sentence, so that his $\Lambda_{\phi, h}$ is not, in general, flip-invariant).  However, $\Lambda_{\phi, h}\cup \Lambda_{\phi, h}=\Lambda^{\pm}_{\phi, h}$, so that $\underset{h\in F_N \smallsetminus \{ 1\}}{\cup} \Lambda^\pm_{\phi,h}=\underset{h\in F_N \smallsetminus \{ 1\}}{\cup} \Lambda_{\phi, h}$ and thus 
the definition of $\Lambda_\phi$ given above agrees with the definition given in \cite{M1}.
\end{rem}

Mitra's main result in \cite{M1}, specialized to the case of mapping tori of hyperbolic automorphisms of free groups, implies:

\begin{thm}\label{thm:fiber}\cite{M1}
Let $\phi\in\Out(F_N)$ be a hyperbolic automorphism,
and let $\hat \iota: \partial
F_N\to\partial G_\phi$ be the Cannon-Thurston map.

Then for $(X,Y)\in \partial^2 F_N$ we have $\hat \iota(X)=\hat \iota(Y)$ if
and only if 
\[
(X,Y)\in \Lambda_\phi\cup\Lambda_{\phi^{-1}}.
\]
\end{thm}

\begin{rem}
\label{diagonally-closed-fiber}
Using Remark \ref{preimage-lamination} we observe that
Theorem \ref{thm:fiber} implies directly that $\Lambda_\phi \cup \Lambda_{\phi^{-1}}$ is an algebraic lamination, and that it is diagonally closed.

However, the fact that each of $\Lambda_\phi$ and $\Lambda_{\phi^{-1}}$ are also laminations (and also diagonally closed) only follows from the following proposition.

\end{rem}

\begin{prop}\label{prop:step1}
Let 
$\phi\in\Out(F_N)$ be an atoroidal iwip. Then \[\Lambda_\phi=L(T_-(\phi))=\diag \left(L_{BFH}(\phi)\right).\]
\end{prop}

\begin{proof}
Recall 
from subsection \ref{trees-graphs}
that $T_-\phi=\frac{1}{\lambda_-}T_-$ with $\lambda_- > 1$, so that for every $h\in
F_N \smallsetminus \{ 1 \}$ and $n \geq 1$ we have:
\[
||\phi^n([h])||_{T_-}=||h||_{T_-\phi^n}=
\frac{1}{\lambda_-^n} ||h||_{T_-} \overset{n \to \infty}{\longrightarrow} 0.
\]
By Definition \ref{Mitra-defn} we have $\Lambda_{\phi, h}^\pm = L(\{\phi^n([h]) \mid n \geq 0 \})$. Thus Lemma \ref{subset-dual-lam} implies $\Lambda^\pm_{\phi, h} \subset L(T_-)$. Since $\Lambda_\phi$ is the union of all $\Lambda^\pm_{\phi, h}$, we deduce:
$$\Lambda_{\phi} \subset L(T_-)$$

From Theorem \ref{prop:KL6} we know that $L_{BFH}(\phi)$ is the only minimal sublamination of $L(T_-)$, so that it has to be contained in any 
sublamination of $L(T_-)$, such as any of the $\Lambda_{\phi, h}^\pm$, and thus in particular in $\Lambda_{\phi}$. 
We obtain:
$$L(T_-) = \diag(L_{BFH}(\phi)) \subset \diag(\Lambda_\phi) 
\subset \diag(L(T_-)) =  L(T_-)$$
Thus all these laminations must be equal. 

By symmetry, we obtain the analogous equalities for $\phi^{-1}$ and $T_+$. From Remark \ref{minimal-distinct} we know that the laminations $L_{BFH}(\phi)$ and $L_{BFH}(\phi^{-1})$ are both minimal, and that they are distinct. Thus Lemma \ref{disjoint-extensions}
implies that $\diag(L_{BFH}(\phi)) \, \cap\,  \diag(L_{BFH}(\phi^{-1})) = \emptyset$. Above we derived $\diag(L_{BFH}(\phi)) = \diag(\Lambda_\phi)$ and $\diag(L_{BFH}(\phi^{-1})) = \diag(\Lambda_{\phi^{-1}})$, so that we have:
$$\diag(\Lambda_{\phi}) \cap \diag(\Lambda_{\phi^{-1}}) = \emptyset$$
By Remark \ref{diagonally-closed-fiber} the union $\Lambda_{\phi} \, \cup\, \Lambda_{\phi^{-1}}$ is diagonally closed, so that we see from Lemma \ref{intersection-total} that both, $\Lambda_{\phi}$ and $\Lambda_{\phi^{-1}}$ must also be diagonally closed. Hence the above derived equality $\diag(\Lambda_{\phi}) = L(T_-) = \diag(L_{BFH}(\phi))$ specifies to 
$$\Lambda_{\phi} = L(T_-) = \diag(L_{BFH}(\phi))\,.$$
\end{proof}

Proposition~\ref{prop:step1} and Theorem~\ref{thm:fiber} immediately imply:

\begin{cor}\label{cor:step1}
Let  $\phi\in\Out(F_N)$ be an atoroidal iwip and let 
$\hat \iota: \partial
F_N\to\partial G_\phi$ be the Cannon-Thurston map.

Then for $(X,Y)\in \partial^2 F_N$ we have $\hat \iota(X)=\hat \iota(Y)$ if
and only if 
\[
(X,Y)\in L(T_+(\phi))\cup   L(T_-(\phi)).
\]
\qed
\end{cor}

We denote by $\sim_\phi$ the equivalence relation defined by the lamination lamination $L_\phi := L(T_+(\phi))\cup   L(T_-(\phi))$, see Definition \ref{laminations-equivalence}.

\begin{prop}
\label{quotient-homeomorphism}
Let  
$\phi\in\Out(F_N)$ be an atoroidal iwip and let $\hat \iota: \partial
F_N\to\partial G_\phi$ be the Cannon-Thurston map.

Then $\partial G_\phi$ is homeomorphic to $\partial F_N/\sim_\phi$, where the latter is considered with the quotient topology. 
\end{prop}

\begin{proof}
We 
apply Proposition \ref{compact-quotient} to L = $L_\phi$ to obtain that $\partial\FN/\sim_\phi$
is a compact Hausdorff topological space. By Corollary~\ref{cor:step1}, the surjective map $\hat \iota: \partial \FN \to \partial G_\phi$ induces a well defined quotient map 
map $r: \partial F_N/\sim_\phi \to \partial G_\phi$, which is by definition continuous and injective, and thus, by the surjectivity of $\hat \iota$, bijective.
Thus $r$
is a continuous bijection between two compact Hausdorff topological spaces $\partial F_N/\sim_\phi$ and $\partial G_\phi$, and therefore $r$ is the desired homeomorphism.
\end{proof}

\begin{prop}
\label{iota-split}
The map $\hat \iota: \partial \FN \to \partial G_\phi$ splits over the maps $\mathcal Q_+ : \partial \FN \to \hat T_+(\phi)$ and $\mathcal Q_- : \partial \FN \to \hat T_-(\phi)$, and thus induces well defined maps
$$\mathcal R_+: \hat T_+(\phi) \to \partial G_\phi
\qquad {\rm and} \qquad
\mathcal R_-: \hat T_-(\phi) \to \partial G_\phi$$
which are surjective, $\FN$-equivariant, and furthermore continuous with respect to both, the metric and the observer's topology on $\hat T_+(\phi)$ and $\hat T_-(\phi)$.
\end{prop}

\begin{proof}
This follows directly from Proposition \ref{quotient-homeomorphism} and Remark \ref{prop:Q} (1) together with the fact that on the trees $\hat T_+(\phi)$ and $\hat T_-(\phi)$ the metric topology is stronger than the observer's topology, with respect to which the maps $\mathcal Q_+$ and $\mathcal Q_-$ are continuous 
(see Proposition \ref{Q-map} (1)). 
\end{proof}

For the sequel we would like to note the following properties of the above defined map $\mathcal R_-: \hat T_-(\phi) \to \partial G_\phi$, where we use the abbreviations $T_+ := T_+(\phi)$ and $T_- := T_-(\phi)$:

\begin{lem}
\label{specification}
(1)
$\mathcal R_-(\bar T_-) \cap \mathcal R_-(\hat T_- \smallsetminus \bar T_-) = \emptyset$.

\smallskip
\noindent
(2)
The restriction $\mathcal R_-|_{\bar T_-}$ of $\mathcal R_-$ to the metric completion of $T_-$ is injective.

\end{lem}

\begin{proof}
Since $\hat \iota = \mathcal R_- \circ \mathcal Q_-$, the only points on which $\mathcal R_-$ is non-injective are the $\mathcal Q_-$-images of the half-leaves of the lamination $L(T_+)$, by Corollary \ref{cor:step1}. But from Proposition \ref{prop:disj} we know that a half-leaf $X$ of $L(T_+)$ cannot satisfy  $L(X) \subset L(T_-)$; thus Proposition \ref{prop:Q} (2) ensures that $\mathcal Q_-(X)$ must lie in $\partial T_-$. This shows both assertions (1) and (2).
\end{proof}

\section{The fibers of the Cannon-Thurston map}
\label{sec:main-results}

Throughout this section we assume that $\phi\in \Out(F_N)$ is an atoroidal iwip, that $\Phi\in \Aut(F_N)$ is a representative of the outer automorphism class $\phi$ and that the mapping torus group
$G_\phi$ is given by presentation $(\clubsuit)$ (see Section~\ref{iwip-autos}) in the generators $F_N, t_\Phi$. 
We will use 
the abbreviations $T_+ := T_+(\phi)$ and $T_- := T_-(\phi)$.
Before starting the proofs of our main results we need to establish some terminology for the boundary points of $G_\phi$:

\smallskip

A point $S \in \partial G_\phi$ is called {\em rational} if it is the fixed point of an element $g \in G \smallsetminus \{1\}$. We write $S = g^\infty$ if $S = \underset{n \to \infty}{\lim} g^n$ (in the topology of the Gromov compactification of hyperbolic groups).

Note  that the $G_\phi$-action on $\partial G_\phi$ induces canonically an action of $\langle \phi \rangle \cong G_\phi/\FN$ on the $\FN$-orbits of points of $\partial G_\phi$. 
We have:

\begin{lem}\label{lem:per}
Let $S \in \partial G_\phi$. Then $S=g^\infty$ for some $g\in G_\phi, g\not\in F_N$ if and only if the $\FN$-orbit of $S$ is $\phi$-periodic.
\end{lem}
\begin{proof}
Suppose that  the $\FN$-orbit of $S$ is $\phi$-periodic. Then there
exist $n\ge 1$ and $w\in F_N$ such that $t_\Phi^n S=wS$. Hence $gS=S$
for $g=w^{-1}t_\Phi ^n$. Since $G_\phi$ is torsion-free word-hyperbolic and $g\ne 1$, the fact that $gS=S$ implies that $S=g^\infty$ or $S=g^{-\infty}$, as required.

Suppose now that $S=g^\infty$ for some $g\in G_\phi, g\not\in F_N$. Thus $g=ut_\Phi^n$ for some $n\ne 0$ and $u\in F_N$.  Then $gS=S$, so that $ut_\Phi^nS=S$ and $t_\Phi^nS=u^{-1}S$. Thus the $\FN$-orbit of $S$ is $\phi$-periodic, as required.
\end{proof}

\begin{defn}
\label{degree}
 Let $S\in \partial G_\phi$.  We define:

\smallskip
\noindent
(1)
The \emph{degree} $\deg(S)$ of $S$ denotes the cardinality of the full preimage of $S$ under the map $\hat \iota: \partial F_N\to \partial G_\phi$.

\smallskip
\noindent
(2)
We define the following {\em classes} of points $S \in \partial G_\phi$:
 \begin{enumerate}
 \item[(i)] the point $S$ is \emph{simple} if $\deg(S)=1$;
 \item[(ii)] the point $S$ is \emph{regular} if $\deg(S) = 2$;
 \item[(iii)] the point $S$ is \emph{singular} if $\deg(S) \geq 3$.
 \end{enumerate}
 \smallskip

\smallskip
\noindent
(3)
We further subdivide the classes of regular and singular points into two {\em types}, as follows:
\begin{enumerate}
\item[(a)]
$S$ is of \emph{$\phi$-type} if
for every two distinct $\hat \iota$-preimages $X,Y\in \partial F_N$ of $S$ we have $(X,Y)\in 
L(T_-)$, and 
\item[(b)]
$S$ is of \emph{$\phi^{-1}$-type} if
for every two distinct $\hat \iota$-preimages $X,Y\in \partial F_N$ of $S$ we have $(X,Y)\in L(T_+)$.
\end{enumerate}
\end{defn}

Notice that, by 
Corollary~\ref{cor:step1} and Proposition~\ref{prop:disj}, 
if $S$ is not simple, then it must either be of $\phi$- or of $\phi^{-1}$-type, so that one obtains:

\begin{prop}
\label{rational-simple}
If $X \in \partial \FN$ is rational, then $\hat \iota(X)$ must be simple.
\qed
\end{prop}

Note that the degree, the class and the type of the points $S$ in $\partial G_\phi$, and also whether or not $S$ is rational, are properties which are invariant under the action of $G_\phi$. This is a direct consequence of the $G_\phi$-equivariance of the map $\hat \iota$.

Thus in particular for every $\FN$-orbit $[S]_{F_N}$ of points $S \in \partial G_\phi$ the {\em degree} is well defined through $\deg([S]_{F_N}) := \deg(S)$.

 \begin{thm}\label{thm:types}
 Let $\phi\in \Out(F_N)$ be an atoroidal iwip  and let $\hat \iota: \partial F_N\to \partial G_\phi$ be the Cannon-Thurston map.
Then 
one has:
\[
\sum
\left( \deg([S]_{F_N})-2\right) \,\,\,\le \,\,\, 2N-2
\] 
where the summation is taken over all $F_N$-orbits $[S]_{F_N}$ of singular points $S$ in $\partial G_\phi$ that are of $\phi$-type.

The same inequality holds if the summation is taken over all $F_N$-orbits $[S]_{F_N}$ of singular points of $\phi^{-1}$-type.

 \end{thm}
 
\begin{proof}
Every singular (or regular) point $S \in \partial G_\phi$ which is of $\phi$-type has by definition as $\hat \iota$-preimage only half-leaves of the lamination $L(T_-)$,
and those
are mapped by $\mathcal Q_-$ to the metric completion $\bar T_-$ (by Proposition \ref{prop:Q} (3)). 
From $\hat i = \mathcal R_- \circ \mathcal Q_-$ it follows that
$S$ must be contained in $\mathcal R_-(\bar T_-)$.

From
Lemma \ref{specification} (1) we know that $\mathcal R_-(\bar T_-) \cap \mathcal R_-(\hat T_- \smallsetminus \bar T_-) = \emptyset$.  
Furthermore we know from Lemma \ref{specification} (2) that 
$\mathcal R_-|_{\bar T_-}$ is injective, so that for any singular $S \in \partial G_\phi$ of $\phi$-type there is a unique $x_S \in \bar T_-$ with $\mathcal R_-(x_S) = S$.
It follows that the $\hat \iota$-fiber of 
$S$ must be equal to the $\mathcal Q_-$-fiber of 
the point $x_S \in \bar T_-$, and hence 
$\deg(S) -2$ must be equal to $\ind(x_S)$.  

Conversely, if $x \in \bar T_-$ has $3$ or more distinct $\mathcal Q_-$-preimages, then those belong to $L(T_-)$ and (again by $\hat i = \mathcal R_- \circ \\mathcal Q_-$) are mapped by $\hat \iota$ to the point $S := \mathcal R_-(x)$, so that $S\in \partial G_\phi$ is a singular point of $\phi$-type, with $x = x_S$ as above.

From the $\FN$-equivariance of $\\mathcal R_-$ it follows that the latter induces a bijection between $\FN$-orbits in $\bar T_-$ and $\FN$-orbits in $\mathcal R_-(\bar T_-) \subset \partial G_\phi$, and hence in particular between $\FN$-orbits of points $x \in \bar T_-$ with $\ind(x) > 0$ and $\FN$-orbits of singular points $S$ in $\partial G_\phi$ that are of $\phi$-type.

Thus we obtain now immediately that the desired inequality is a direct consequence of the $\mathcal Q$-index formula of Coulbois-Hilion \cite{CH10}, see Theorem \ref{prop:Q-ind}.
\end{proof}

As consequence we obtain a number of interesting insights:

\begin{thm}\label{thm:mainresult}
Let $\phi\in \Out(F_N)$ be an atoroidal iwip  and let $\hat \iota: \partial F_N\to \partial G_\phi$ be the Cannon-Thurston map. Then the following holds:
\begin{enumerate}
\item
For every $S\in \partial G_\phi$ we have:
$$\deg(S)\le 2N$$
\item
The number of $F_N$-orbits of singular points of $\phi$-type  (respectively of $\phi^{-1}$-type) in $\partial G_\phi$ 
satisfies:
$$ {\rm card}\{\FN\cdot S \subset \partial G_\phi \mid S \,\, {\rm singular\,\,of\,\,} \phi{\rm-type}\} \leq 2N -2$$
\item
Every singular point $S\in \partial G_\phi$  is rational. More precisely, there exists $g\in G_\phi \smallsetminus \FN$  such that $S=g^\infty$.  
\end{enumerate}
\end{thm}

\begin{proof}
Assertions (1) and (2) are direct consequences of the inequality stated in Theorem \ref{thm:types}.  

By Theorem \ref{thm:types},  there are only finitely many $\FN$-orbits of singular points, so that necessarily each of them must be periodic under the action of $\phi$.
Hence assertion (3) of Theorem~\ref{thm:mainresult} follows from Lemma~\ref{lem:per}.

\end{proof}

\begin{prop}
\label{double-positive}
Let $S$ and $g$ be as in Theorem \ref{thm:mainresult} (3). If $S$ is of $\phi$-type then 
$g$ must be 
of the form $g = wt_\Phi^m$, with $w\in F_N$ and $m \geq 1$. 

Similarly, if $S$ is of $\phi^{-1}$-type then  $g$  is 
of the form
$g = v t_\Phi^m$, with $v\in F_N$ and $m \leq 1$.
\end{prop}

\begin{proof}
From the argument given in the proof of Theorem \ref{thm:mainresult} we see that $S$ is the fixed point of some element $g \in G_\phi$ of the form $g = wt_\Phi^m$, with $w\in F_N$ and $m \neq 1$.

In particular, we can assume, by possibly replacing $g$ by its inverse, that $m \geq 1$, so that
$g$ acts on $T_-$ as homothety $H_g$ with stretching factor $\lambda_g < 1$ (compare subsection \ref{trees-graphs}), 
and it has a unique fixed point $P_-(g)\in \bar T_-$. 
From the assumption that $S$ is of $\phi$-type, i.e. the $\hat \iota$-preimage of $S$ are half-leaves of $L(T_-)$, we obtain (using Proposition \ref{prop:Q} (3)) that $\mathcal Q_-(\hat \iota^{-1}(S)$ is contained in $\bar T_-$. 
It follows from the $G_\phi$-equivariance of the map $\mathcal R_-$ and the injectivity of its restriction to $\bar T_-$ (Lemma \ref{specification} (2))
that $\mathcal R_-(P_-(g)) = S$.

We now consider any point $Z \in \bar T_-$ which is distinct from $P_-(g)$, and hence (since $H_g$ is a homothety) not fixed by $g$. Since the stretching factor of $H_g$ satisfies $\lambda_g <1$, it follows that $\underset{n \to \infty}{\lim} g^n Z = P_-(g)$.  Hence it follows from the $G_\phi$-equivariance and the continuity of $\mathcal R_-$ that $g^n(\mathcal R_-(Z))$ converges towards $\mathcal R_-(P_-(g)) = S$: This implies $S = g^\infty$, since $g Z \neq Z$ and hence $\mathcal R_-(Z) \neq g(\mathcal R_-(Z))$, by 
Lemma \ref{specification} (2).
\end{proof}

 \begin{cor}\label{cor:4N-5}
 For any atoroidal iwip $\phi \in \Out(\FN)$ we have:
 \[
\sum_{[S]_{F_N}}
  \left(\deg([S]_{F_N})-2\right)\,\,\,\le\,\,\, 4N-5
 \]
where the sum is taken over all $F_N$-orbits $[S]_{F_N}$ of singular points $S$ in $\partial G_\phi$.
 
Moreover, the number of $F_N$-orbits of singular points in $\partial G_\phi$ is bounded above by $4N-5$.
 \end{cor}
 
\begin{proof}
By splitting the sum on the left of the claimed inequality into two partial sums, one for all $S$ of $\phi$-type, and one for all $S$ of $\phi^{-1}$-type, we obtain directly from the inequality of Theorem \ref{thm:types} the upper bound $4N-4$ on the right hand side of the inequality. However, the only way to get equality would be if both of the above partial sums add up to $2N-2$. But this happens if and only if both trees $T_+$ and $T_-$ are geometric (see 
\cite{CH12}), which in turn implies (see \cite{Gui3,CH12}) that $\phi$ is induced by a homeomorphisms of a surface with boundary, contradicting the assumption that $\phi$ is atoroidal (see Remark \ref{hyp-ator-surf}).

The bound on the number of orbits of singular points is an immediate consequence of this inequality, since each such orbit has degree $\geq 3$.
\end{proof}

\begin{rem}\label{lem:2sing}
It follows from Remark \ref{laminations-folk} (1) (or alternatively, from using the action of $G_\phi$ on the attracting tree $T_+$ rather than on $T_-$)
that
there exists at least one singular point of $\phi$-type and at least one singular point of $\phi^{-1}$-type in $\partial G_\phi$. 
 In particular, 
 there exist at least 2 distinct 
 $\FN$-orbits of singular points in $\partial G_\phi$.
 \end{rem}

 \begin{thm}\label{thm:4N-1}
 Let $\phi\in \Out(F_N)$ be an atoroidal iwip, 
 let $\hat \iota: \partial F_N\to \partial G_\phi$ be the Cannon-Thurston map,
and let $g \in G_\phi \smallsetminus \{1\}$ be arbitrary.
 
Then
\[
\deg(g^\infty) +\deg(g^{-\infty}) \le 4N-1. 
\]

 \end{thm}

\begin{proof}
If at least one of $g^\infty$ or $g^{-\infty}$ is simple or regular, the inequality follows directly from Theorem \ref{thm:mainresult} (1).
Otherwise 
we obtain from 
Proposition \ref{double-positive} that precisely one of $g^\infty, g^{-\infty}$ is of $\phi$-type and one is of $\phi^{-1}$-type, and hence they can not belong to the same $G_\phi$-orbit. Hence the asserted inequality is a direct consequence of Corollary \ref{cor:4N-5}.
\end{proof}

\begin{rem}
\label{sharp-bound}
The upper bounds given in Theorem \ref{thm:types}, Theorem \ref{thm:mainresult} (1), Corollary \ref{cor:4N-5} and  Theorem \ref{thm:4N-1} are sharp:  A concrete example, for every $N \geq 3$, 
where for each of these statements the given inequality is actually an equality,
has been worked out in \cite{JL}. 

The same examples show also that the ``lower bound'' given in Remark \ref{lem:2sing} is sharp:  In these examples there is only one $\FN$-orbit of singular point of $\phi$-type and only one of $\phi^{-1}$-type.

Examples for $G_\phi$ with only one $\FN$-orbit of singular points $S$ of $\phi$-type, with $\deg (S) = 3$ have been worked out by C. Pfaff, for the case $N = 3$ (see 
\cite{Pfaff}).
\end{rem}

\medskip

Recall that for any non-elementary hyperbolic group $G$ the Gromov boundary $\partial G$ has uncountable cardinality. Since $G$ is finitely generated and hence countable, it follows that there are uncountably many $G$-orbits $G\cdot S$ of points $S \in \partial G$.

\begin{prop}\label{prop:reg}
Let 
$\phi\in \Out(F_N)$ be an atoroidal iwip. Then there are uncountably many simple points in $\partial G_\phi$ 
and hence
also uncountably many $G_\phi$-orbits of such simple points.
\end{prop}

\begin{proof}
From Lemma \ref{equal-half-leaves} and Theorem \ref{prop:KL6} we know for 
the zero laminations $L(T_+)$ and $L(T_-)$ 
that 
their sets of half-leaves satisfy 
$L^1(T_-)$ $ = L^1_{BFH}(\phi)$ and  $L^1(T_+) = L^1_{BFH}(\phi^{-1})$. Since $L_{BFH}(\phi)$ and $L_{BFH}(\phi^{-1})$ are minimal (see Remark \ref{minimal-distinct}) we can apply Remark 
\ref{laminations-folk} (2)
to obtain that the complement $\partial\FN \smallsetminus (L^1(T_+) \cup L^1(T_-))$ is uncountable. It follows from 
Corollary \ref{cor:step1}
that all of these complementary points are mapped by $\hat\iota$ to distinct points of $\partial G_\phi$, and that those are all simple.
\end{proof}

\begin{prop} \label{prop:semi}

The set of regular points 
in $\partial G_\phi$ is uncountable.
In particular,
there are uncountably many $G_\phi$-orbits of regular points
in  $\partial G_\phi$.
\end{prop}

\begin{proof}

From Remark \ref{minimal-distinct} (1) and Remark \ref{laminations-folk} (3) we know that both, $L(T_+)$ and $L(T_-)$, are uncountable sets, and hence there are uncountably many $\FN$-orbits in each of them. From Theorem \ref{thm:mainresult} we know that there are only finitely many $\FN$-orbits of singular points in $\partial G_\phi$, and that their degree is bounded by $2N$. Hence it follows from Corollary \ref{cor:step1} that there are uncountably many regular points in $\partial G_\phi$, and hence also uncountably many $\FN$-orbits.
\end{proof}

\appendix
\section{Mitra's results on quasiconvexity of subgroups in hyperbolic free-by-cyclic groups}\label{A}

Let $L$ be an algebraic lamination on $F_N$ and let $H\le F_N$ be a finitely generated subgroup. Thus $H$ is quasi-isometrically embedded in $F_N$ and hence $\partial H\subseteq \partial F_N$.
Following~\cite{BFH97}, we say that a leaf $(X,Y)\in L$ is \emph{carried} by $H$ if there exist $w\in F_N$ and $X',Y'\in \partial H$ such that $(X,Y)=w(X',Y')$.
We say that $L$ is \emph{minimally filling} in $F_N$ if no leaf of $L$ is carried by a finitely generated subgroup of infinite index in $F_N$.

Proposition~\ref{prop:step1}  shows that for a hyperbolic iwip $\phi\in \Out(F_N)$ we have $\Lambda_\phi=L(T_-)=\diag \left(L_{BFH}(\phi)\right)$ and $\Lambda_{\phi^{-1}}=L(T_+)=\diag \left(L_{BFH}(\phi^{-1})\right)$. This relationship between $\Lambda_{\phi^{\pm 1}}$ and $L_{BFH}(\phi^{\pm 1})$ is more delicate than one might suspect upon initial examination of the definitions of these objects, and there do exist some incorrect claims on this topic in the literature. 

Thus in a 1999 article~\cite{Mitra99} Mitra mistakenly claims, with a reference to Proposition~1.6 in \cite{BFH97},  that $\Lambda_\phi=L_{BFH}(\phi)$ and $\Lambda_{\phi^{-1}}=L_{BFH}(\phi^{-1})$; that mistake is based on misreading the definition of weak convergence (Definition~1.5 in \cite{BFH97}) and consequently misapplying Proposition~1.6 of \cite{BFH97}. The mistaken claim that  $\Lambda_{\phi^{\pm 1}}=L_{BFH}(\phi^{\pm 1})$ is then used in the proof of one of the main results of \cite{Mitra99}, Theorem~3.4 there:

\begin{thm}\label{T:Mitra99}\cite{Mitra99}
Let $\Phi\in \Aut(F_N)$  be a hyperbolic iwip and let $G_\Phi=F_N\rtimes_\Phi \mathbb Z$ (so that  $G_\Phi$ is word-hyperbolic). Then a finitely generated subgroup $H_1$ of $F_N$ is quasiconvex in $G_\Phi$ if and only if $H_1$ has infinite index in $F_N$. 
\end{thm}

Since, as noted above, $L_{BFH}(\phi^{\pm 1})$ are contained in but not equal to $\Lambda_{\phi^{\pm 1}}$, this creates a gap in the proof of Theorem~3.4 given in \cite{Mitra99}. 
This gap can be fixed, using, for example, Proposition~\ref{prop:step1}, in the following way. To obtain Theorem~3.4 in \cite{Mitra99}, Mitra uses Theorem~3.3 in \cite{Mitra99}, whose proof does go through if one knows that for a hyperbolic iwip $\phi\in\Out(F_N)$ each $\Lambda_{\phi^{\pm 1}}$ is minimally filling in $F_N$.  Proposition~2.4 in \cite{BFH97} shows that for any iwip $\phi\in\Out(F_N)$ the laminations $L_{BFH}(\phi^{\pm 1})$ are minimally filling in $F_N$. But, as noted above, since laminations $\Lambda_{\phi^{\pm 1}}$ are bigger than $L_{BFH}(\phi^{\pm 1})$ and have more leaves than the latter, Proposition~2.4 in \cite{BFH97}  does not directly imply that $\Lambda_{\phi^{\pm 1}}$ are minimally filling in $F_N$ also.

We can show that   $\Lambda_{\phi^{\pm 1}}$ are minimally filling in $F_N$, thereby fixing the proofs of Theorems~3.3 and 3.4 in \cite{Mitra99}, in a couple of different ways:
\begin{prop}\label{prop:algfill}
Let $\phi\in \Out(F_N)$ be a hyperbolic iwip. Then the laminations $\Lambda_{\phi}$ is minimally filling in $F_N$.
\end{prop}

\begin{proof}
The proof of Proposition~2.4 in \cite{BFH97}  goes through verbatim for half-leaves of $L_{BFH}(\phi)$ (in the sense of Definition~\ref{defn:algebraic_lamination}). This proof (see also the proof of Proposition~4.6 in \cite{Ka}) shows that if $(X,Y)\in L_{BFH}(\phi)$ and $H_1\le F_N$ is a finitely generated subgroup of infinite index, then there do not exist $w\in F_N$, $Y'\in \partial H_1$ such that $wY'=Y$.
Suppose now that $(X,Y)\in \Lambda_\phi$ is a leaf of $\Lambda_\phi$ such that $(X,Y)$ is carried by a finitely generated subgroup $H_1$ of infinite index in $F_N$, that is $(X,Y)=w(X',Y')$ for some $w\in F_N$ and $X',Y'\in \partial H_1$. Since, by Proposition~\ref{prop:step1}, $\Lambda_\phi=\diag \left(L_{BFH}(\phi)\right)$, it follows that there exists $X_1\in \partial F_N$ such that $(X_1,Y)\in L_{BFH}(\phi)$.  Since $Y=wY'$ and $Y'\in \partial H_1$, we get a contradiction with the modified "half-leaf" version of Proposition~2.4 in \cite{BFH97}  stated above. Hence $\Lambda_{\phi}$ is minimally filling, as required.

Another way to see that $\Lambda_{\phi}$ is minimally filling is via a recent general result of Reynolds~\cite{Rey11}. In \cite{Rey11} he proves that if $T\in\cvnbar$ is a free $F_N$-tree which is "indecomposable" (in the sense of Guirardel~\cite{Gui3}), then $L(T)$ is minimally filling. It is well-known that for a hyperbolic iwip $\phi$ the trees $T_\pm$ are $F_N$-free; as shown recently by Coulbois and Hilion~\cite{CH12}, the trees $T_\pm$ are also indecomposable. Hence $L(T_\pm)$ are filling. The proof of Proposition~\ref{prop:step1} above shows that establishing the inclusion $\Lambda_\phi\subseteq L(T_-)$ (rather than actual equality) is fairly straightforward and does not require invoking Proposition~\ref{prop:KL6}. Thus $\Lambda_\phi\subseteq L(T_-)$  and since $L(T_-)$ is minimally filling, it follows that $\Lambda_\phi$ is minimally filling as well.
\end{proof}

The above arguments fill the gap in the proof of Theorem~3.4 in \cite{Mitra99}.  See an updated and corrected (September 2012) version~\cite{M5} of Mitra's 1999 paper~\cite{Mitra99} for additional details.

\affiliationone{Ilya Kapovich\\Department of Mathematics\\ University
  of Illinois at Urbana-Champaign\\ 1409 West Green Street\\ Urbana, IL 61801, U.S.A.
   \email{kapovich@math.uiuc.edu}}
\affiliationtwo{Martin Lustig\\LATP,
Centre de Math\'ematiques et Informatique\\
Aix-Marseille Universit\'e\\
39, rue F.~Joliot Curie\\ 
13453 Marseille 13, France
   \email{Martin.Lustig@univ-amu.fr}}

\end{document}